\documentclass[11pt,a4paper,reqno]{amsart}
\usepackage{amssymb}
\usepackage{amscd}
\usepackage{amsmath}
\usepackage{amsfonts}
\usepackage{algorithm}

\usepackage{cite}

\theoremstyle{plain}
\newtheorem{theorem}{Theorem}[section]

\newtheorem{lemma}[theorem]{Lemma}

\newtheorem{proposition}[theorem]{Proposition}

\theoremstyle{definition}
\newtheorem{definition}[theorem]{Definition}

\def\Shape(#1){\operatorname{Shape}(#1)}

\newcommand{\magma}{{\sc Magma}}
\newcommand{\Z}{\mbox{\bf Z}}
\newcommand{\Q}{\mbox{\bf Q}}
\newcommand{\F}{\mbox{\bf F}}

\newcommand{\C}{\mbox{\bf C}}
\newcommand{\Ad}{A^\dagger}
\newcommand{\AdK}{A^\dagger\otimes_R K}
\newcommand{\Fpd}{F_p^{\dagger}}
\newcommand{\Prj}{\mathbb P}
\newcommand{\myim}{\mbox{Im}}
\begin{document}
\title{An extension of Kedlaya's algorithm for hyperelliptic curves.}
\author{Michael Harrison} 
\address{School of Mathematics and Statistics F07, University of Sydney, NSW 2006, Australia}
\begin{abstract}
In this paper we describe a generalisation and adaptation of Kedlaya's algorithm for
computing the zeta function of a hyperelliptic curve over a finite field
of odd characteristic that the author used for the implementation of the algorithm
in the \magma\ library. We generalise the algorithm to the case of an even
degree model. We also analyse the adaptation of working with the $x^idx/y^3$
rather than the $x^idx/y$ differential basis. This basis has the computational
advantage of always leading to an integral transformation matrix whereas the
latter fails to in small genus cases. There are some theoretical subtleties that
arise in the even degree case where the two differential bases actually lead to 
different redundant eigenvalues that must be discarded.
\end{abstract}
\maketitle

\section{Introduction}
Kedlaya's algorithm for hyperelliptic curves in odd characteristic was one of
the first practical computational algorithms for computing the zeta function
of a curve of genus greater than 1 over a large finite field of small
characteristic \cite{Ked01}, \cite{Ked04}. It was generalised by Denef and
Vercauteren to characteristic two \cite{DV1_06} and has also been
extended to more general curves like $C_{ab}$ curves \cite{DV2_06}.
Kedlaya's algorithm is based on the calculation of the Frobenius action on an
appropriate $p$-adic cohomology group that can be described in sufficiently concrete 
terms for explicit computer computations to be made. In the hyperelliptic case,
Kedlaya used Monsky-Washnitzer cohomology on the open affine subset of the curve
defined by the removal of all Weierstrass points.

In 2003, the author wrote the implementation of Kedlaya's algorithm in the standard
user library of
the \magma\ computer algebra system \cite{Mag97}. In practical terms, there 
appeared to be two issues with the algorithm as it stood.

Firstly, it only covered the odd degree case of a hyperelliptic model with a 
single 
point at infinity. Following Kedlaya's analysis, we extended the 
algorithm in a natural way to also cover the even degree case. The extension is
fairly straightforward and the algorithm runs as before except that a degree one
term has to be removed from the final characteristic polynomial corresponding to
an extra eigenvalue $q$ (the field size) arising from the extra
point at infinity removed from the complete curve.

More seriously, if $p$, the characteristic of the base finite field $\F_q$, is 
small compared to the genus $g$ of the hyperelliptic curve $C$ - specifically if
$p \le 2g-1$ in the odd degree case and $p \le g$ in the even degree case -
then the matrix $M$ representing the $\sigma$-linear transformation of
$p$-Frobenius on Kedlaya's chosen differential basis of cohomology is non-integral.
That is, the integral lattice generated by the basis is not stable under Frobenius.
Because $M$ has to be $\sigma$-powered to a large degree to get to the
final result, this presents obvious $p$-adic precision problems. If $M$ represented
a {\it linear} transformation, it could be easily replaced by an integral conjugate
before powering (though then the characteristic polynomial of the power could
be computed without matrix powering, anyway!), but the semi-linear situation is
not so easy to work with. This issue is remarked upon in \cite{Ked03} and can be 
dealt with in a number of ways. One approach is to try to analyse the situation
using high-powered techniques like crystalline cohomology (or general $F$-module
theory) to find an integral lattice to work with that is invariant under Frobenius.
For example, Edixhoven gives a general criterion
for stability under Frobenius of a sub $\Z_p$-module $L$ of the $\Z_p$-module of
differentials generated by Kedlaya's differential basis in Prop. 5.3.1 of \cite{Edi06}.
A full proof of the criterion can be found in \cite{Bog08}.
See also \cite{CDV06}, which is described further below, for more general plane curves.
\smallskip
 
In this very concrete situation, however, we computed that a slightly different 
differential basis for the minus part of the $H^1$ cohomology always works: namely 
differentials of the form $x^idx/y^3$ rather than $x^idx/y$.
The computation is again straightforward but, as far as we are aware, it has not 
appeared
in detail before in the literature so, for completeness, we will show that 
Kedlaya's 
reduction process applied to this space of differentials always leads to an 
integral matrix $M$.

The interesting technical point is that the $y^3$ differentials only form a basis
for the minus part of the cohomology in the odd degree case. In the even degree 
case,
the map from this space of differentials into $H^1_-$ actually has a 1-dimensional
kernel and cokernel. It turns out that the kernel has eigenvalue $1$ and cokernel
eigenvalue $q$ under $q$-Frobenius, so in this case we have to remove a factor of
$t-1$ rather than $t-q$ from the final characteristic polynomial. This is 
demonstrated in the final section of the paper.
\medskip

In summary, in the even degree case, one additional eigenvalue of Frobenius occurs on the
affine Monsky-Washnitzer cohomology because of the additional point removed at infinity.
This merely has to be removed at the end in order to get the numerator of the zeta function.
Our alternative set of differentials generate a $\Z_p$-module $V$ with Frobenius action.
This space genuinely gives an F-stable lattice in $H^1_-$ for odd degree and the algorithm goes
through as before, except with guaranteed $p$-integral matrices. For even degree, $V$ also
gives $p$-integrality but $V\otimes\Q$ doesn't quite coincide with $H^1_-$ as a Frobenius module.
However, an explicit analysis in this case shows that the difference between $V$ and $H^1_-$
results in just having to remove a different additional eigenvalue at the end. 
\medskip

A very general Kedlaya-style algorithm applying to non-degenerate plane curves is 
presented by
Castryck, Denef and Vercauteren in \cite{CDV06}. There, a deterministic algorithm is
given where a basis for cohomology is determined and an integrality analysis is
performed involving Edixhoven's criterion and consideration of the Newton polygon of the curve. 
The hyperelliptic
case, however, with its particular choices of differential bases, is still
an important special case amenable to the specific original analysis of Kedlaya
and that presented here, and I have had requests from a number of people to
publish details of the \magma\ implementation.
\medskip

We should also mention some of the other point-counting methods which have been
developed over the last decade for curves of genus greater than 1 and that use different techniques
to that of Kedlaya.

Generalising the elliptic curve case, Mestre devised
an algorithm for ordinary hyperelliptic curves in characteristic $2$ based on the theory of the 
{\it canonical lift}. This computes a $2$-adic approximation of a particular function
of the eigenvalues of Frobenius from which a finite number of possibilities for the
characteristic polynomial of Frobenius can be obtained by rational reconstruction in many
cases (e.g. when the Jacobian is irreducible). Again generalising their algorithm for the genus 1 case, 
Lercier and Lubicz found a way to efficiently effect the lifting stage to obtain a
quasi-quadratic algorithm \cite{LeLu06}. The author implemented this algorithm for the standard
\magma\ user library. Following the work of Robert Carls on theta structures of canonical lifts \cite{Car07}, 
Carls and Lubicz have generalised the algorithm to odd characteristic \cite{CaLu09}.

Another important $p$-adic method is the {\it deformation} method of Lauder and Wan \cite{LaWa08}.
This generalises from the curve case to higher-dimensional hypersurfaces and provides the basis
for the computation of zeta-functions of fairly general varieties over finite fields. The ideas go
back to Dwork and use his approach to $p$-adic cohomology theory, working with parametrised families of  
hypersurfaces and continuously deforming to ones of special form (diagonal in Dwork's original work).
R. Gerkmann has further studied the method, considering relations to rigid cohomology and practical
$p$-adic precision analysis \cite{Ger07}. He has written an implementation in \magma. Fuller details
for the deformation method in the particular case of hyperelliptic curves have been worked out by
H. Hubrechts \cite{Hub08} who provided the implementation that appears in the standard \magma\ user library.
\medskip

A brief outline of the paper is as follows. In the next section, we introduce basic notation,
summarise Kedlaya's original algorithm and describe our extension of it. We also give a brief
overview of Monsky-Washnitzer cohomology and explain Kedlaya's reduction procedure on
differentials which remains formally the same in the extended version.

In Section 3, we consider our alternative (pseudo)-basis and give a proof of the
integrality of the reduction of the image of $p$-Frobenius on its elements alongside an
analysis of Kedlaya's original basis. We also give the short proof of the generalisation
of the point-counting formula to even degree hyperelliptic models.

Finally, in the last section we give proofs of the slightly more technical result relating
the space spanned by our pseudo-basis to its image in Monsky-Washnitzer cohomology and
giving the difference between the eigenvalues of Frobenius on these two spaces.  
\bigskip

\noindent{\bf Acknowledgements\ }  I would like to thank the referees for many useful suggestions.
\bigskip
                  
\section{Review of Kedlaya's Algorithm}

In this section we give a summary of Kedlaya's algorithm as it appears in 
\cite{Ked01} as well as describing our extension of it. The basic notation introduced below
will be used throughout the paper.
\bigskip

\noindent\underline{Basic notation.}
\medskip

Throughout, $q := p^n$ will denote a positive power of an odd prime $p$.
$k$ will denote the finite field $\F_q$, unless otherwise indicated. $R$
will denote $W(k)$, the ring of integers of $K$, the unique unramified
degree $n$ extension of the local field $\Q_p$. $\sigma_p$ will denote the
$p$-Frobenius automorphism of $R$ or $K$ that reduces to $a \mapsto a^p$
on $k$.

$C$ will denote the hyperelliptic curve which is the projective
normalisation of the smooth plane affine curve $C_1$ with defining equation
$$ y^2 = Q(x) $$
where $Q(x) = a_dx^d+\ldots +a_0$ is a separable polynomial of degree $d$ in
$k[x]$. To simplify notation, we also use $Q(x)$ to denote some arbitrary
lift of $Q(x)$ to $R$ (i.e. a degree $d$ polynomial over $R$ such that reduction
mod $p$ of the coefficients gives $Q(x)$). It will always be clear from the
context which polynomial is being referred to.

We let $g$ denote the genus of $C$, so that $d=2g+1$ or $d=2g+2$. We refer to
the $d=2g+1$ case as the odd case and the $d=2g+2$ case as the even case. In
the odd case $C\backslash C_1$ consists of a single $k$-rational point, which is
a Weierstrass point of $C$ and will sometimes be referred to as $\infty$.
In the even degree case, $C\backslash C_1$ consists of a pair of non-Weierstrass
points, $\infty_1$ and $\infty_2$, which are either $k$-rational or conjugate
points over $\F_{q^2}$. Computationally, it is easiest to transform the initial
$Q$ over $k[x]$ so that $a_d = 1$ and the lift to $R$ of $a_d$ is also $1$. This
may involve working with the quadratic twist of $C$ in the even case, but there
is no problem converting back the final result (by substituting $t \mapsto -t$ in
the numerator of the zeta function). So from now on, we assume that $a_d$ is $1$
and $C$ has two $k$-rational points at infinity in the even case.

Following Kedlaya, we define $C^{a}$ as the open affine subset of $C_1$
given by inverting $y$; i.e. $C^{a} = Spec(A_k)$ where
$$ A_k := k[x,y,y^{-1}]/(y^2-Q(x)) $$
and we will let
$A_R := R[x,y,y^{-1}]/(y^2-Q(x))$
which is a finitely-generated, $R$-smooth $R$-algebra with
$A_R\otimes_R k \simeq A_k$. $C^{a}$ is just $C$ with all Weierstrass points
and points at infinity removed.
\bigskip

\noindent\underline{Basic outline of the algorithm.}
\medskip

Given an odd degree model of a hyperelliptic curve $C$ over $\F_q$ as above, Kedlaya's 
algorithm computes the degree $2g$ monic polynomial $L(X)$ that gives the numerator
of the zeta-function of $C$ [$\zeta_C(s) = L(q^{-s})/(1-q^{-s})(1-q^{1-s})$]. The
number of points on $C$, $\#C(\F_{q^r})$, or the order of its Jacobian,
$\#Jac(C)(\F_{q^r})$, over any finite extension $\F_{q^r}$ of the base field can be
simply computed from $L(X)$ in the usual way (e.g. see Appendix C \cite{H77}).

The main stages are given in Algorithm 1.
\medskip

%\fbox{\parbox{15cm}{
%\hrule
%\hrule
%\smallskip
\begin{algorithm}
\caption{Kedlaya's original algorithm}

\begin{description}
\item[Step 0] Input $Q(x)$.
\item[Step 1] Working in $W(k)[x][[1/y]]$, compute $(1/y^\sigma)$ to sufficiently
large $p$-adic and $(1/y)$-adic precision by formally expanding
$$ y^{-p}\left(1+{{Q(x)^\sigma-Q(x)^p}\over{y^{2p}}}\right)^{-1/2}$$
This gives a finite approximation of the image of the differential
basis of cohomology $x^i(dx/y), 0 \le i \le d-2$ under $p$-Frobenius.
\item[Step 2] Reexpress the images of the differentials as $W(k)\otimes\Q$-linear
combinations of the differential basis using the {\tt RednA} and {\tt RednB}
reduction processes described below. This results in a $(2g)$-by-$(2g)$
matrix $M$ for the action of $p$-Frobenius to finite $p$-adic approximation.
\item[Step 3] By binary-powering or similar, compute the product $N=MM^\sigma\ldots
M^{\sigma^{n-1}}$ and the characteristic polynomial $F_p(X)$ of $N$.
This is actually equal to $L(X) \in \Z[X]$ but will have been determined
here in $\Z_p[X]$ to a large, finite $p$-adic precision.
\item[Step 4] Recover and return $L(X)$ from the $p$-adic approximation in step 3,
using the Weil bound to guarantee correct integer coefficients.
\end{description}

\end{algorithm}
%\smallskip

%\hrule
%\hrule
%}}
\bigskip

\noindent\underline{Extension of the original algorithm.}
\medskip

We adapt/extend the original algorithm in two ways.
\begin{itemize}
\item Even degree models ($d$ even) are allowed.
\item When $p < 2g$, $d$ odd, or $p \le g$, $d$ even, the differential pseudo-basis
$x^i(dx/y^3)$, $0 \le i \le d-2$ is used rather than $x^i(dx/y)$.
\end{itemize}

The first change extends the algorithm to an arbitrary hyperelliptic
curve (possibly after applying a quadratic twist as described earlier).

The second change guarantees that we always work with a $p$-Frobenius matrix $M$ with $p$-integral
coefficients. In the cases where we use the alternative differential basis, Kedlaya's original 
basis generally leads to a $p$-adically non-integral $M$. Strictly speaking, the alternative set
of differential forms only form a basis for cohomology when $d$ is odd. This is why we refer to
it as a pseudo-basis. It still leads to correct results in the new algorithm.
All of this, along with the justification for 
the new Step 4 below in the even degree case, is demonstrated in Sections 3 and 4.
\bigskip

The new algorithm is formally very similar to the original, so we will just state the changes
that need to be made.
\bigskip

\noindent{\bf Steps 1 and 2}. These are unaffected except that the
expression to be formally expanded in step 1 has exponent $-3/2$
rather than $-1/2$ in the cases where the alternative differential pseudo-basis is used.
The matrix $M$ in step 2 will be of size $2g+1$ rather than $2g$ when $d$ is even.
\medskip

\noindent{\bf Step 4} Compute $L_1(X)$ from the $p$-adic approximation to $F_p(X)$ coming
from Step 3. If $d$
is even then let $L(X)=L_1(X)/(X-q)$ if using the $dx/y$ basis or $L(X)=L_1(X)/(X-1)$
if using the $dx/y^3$ pseudo-basis. If $d$ is odd, just let $L(X) = L_1(X)$.
Return $L(X)$.
\bigskip

The linear factor that has to be removed in the even case comes from an extra
eigenvalue of the action of Frobenius on cohomology (see Section~\ref{sec_zeta}).
That the factor is different for the pseudo-basis comes from the relation between it and
an actual cohomology basis. The extra $q$ eigenvalue is lost but a new eigenvalue $1$
appears (see Section 4).
\medskip

In practice, only half of the coefficients of $F_p(X)$ (those of the higher powers of $X$)
need to be computed (because of the $\alpha \leftrightarrow q/\alpha$ symmetry of the 
algebraic roots of $L(X)$) and we can effectively remove the extra $X-q$ or $X-1$ factor
from $F_p(X)$ (rather than from $L_1(X)$ at the end) in the even degree case during these 
computations. The coefficients can be computed from the traces of the first $g$ powers
of $N$ as a matrix over $\Z_p$. Removing the extra factors at this stage means that there
is no necessity to increase the $p$-adic precision to which we need to know $N$ beyond the
same lower bound used in the odd-degree case.  This is determined from the upper bound for the
size of the (top) coefficients of $L(X)$ that comes from all of its roots (over $\C$)
having absolute value $\sqrt{q}$. Expressions in $g$ and $n$ for the $p$-adic precision needed
in the initial series expansion computed in Step 1 are given near the end of Section 3.
\medskip

In the remainder of this section - which relates to Steps 1 and 2 - where we describe the
Monsky-Washnitzer cohomology groups and the reduction procedures for Step 2, no distinction need
be made between the even and odd degree cases except where indicated.

That the differential reductions of Step 2 take $p$-Frobenius transforms of elements of the
pseudo-basis back into linear combinations of such elements will be demonstrated in
Lemma~\ref{p_div2}.
\bigskip

\noindent\underline{Monsky-Washnitzer Cohomology.} \cite{MW68}, \cite{Mon68},
\cite{Mon71}.
\medskip

Let $X$ be a non-singular affine scheme over $k$. Monsky
and Washnitzer defined a $p$-adic cohomology theory for such $X$ with appropriate
fixed-point theorems for proving zeta-function results. Kedlaya used this
(originally at least) to provide the technical basis for his algorithm.
Monsky-Washnitzer cohomology agrees with Berthelot's more
general rigid cohomology in the affine case and is pleasantly explicit in its definition.
We will need some of its properties for later proofs and
so we give a brief description of the theory here.
\bigskip

\noindent Let $A_k$ temporarily represent the affine coordinate ring of our general $X$
and $A_R/R$ be a lift to an $R$-smooth $R$-algebra as above and $A_K = A_R\otimes_R K$.
\medskip

\begin{definition}
Let $F_q$ be the $k$-linear Frobenius endomorphism $A_k \rightarrow A_k$ given
by $a \mapsto a^q$.
Similarly, let $F_p$ be the $k$-semilinear endomorphism of $A_k$, $a \mapsto a^p$.
\end{definition}
\medskip

\noindent The goal is to define a good cohomology group which simulates de Rham cohomology
of $A_K$ and to which $F_q$ lifts. $F_q$ lifts to the $p$-adic completion,
$\hat{A}_R = \lim_{\stackrel{\leftarrow}{n}} A_R/p^nA_R$, but the
de Rham cohomology of $\hat{A}_K = \hat{A}_R\otimes_R K$ is usually bigger than
that of $A_K$. Monsky-Washnitzer define a subalgebra $A^{\dagger}$ of $\hat{A}_R$, referred
to as the weak completion, as follows. If $x_1,\ldots,x_r$ are $R$-algebra
generators of $A_R$ then 
$$ A^{\dagger} := \{ \sum^{\infty}_{n=0} a_n p_n(x_1,\ldots,x_r) : a_n \in p^nR,
   \ p_n \mbox{ of total degree } \le C(n+1) \mbox{ for some } C > 1 \} $$
and $A^{\dagger}_K = A^{\dagger}\otimes_R K$. Up to isomorphism, $A^\dagger$ is shown
to be independant of the lift $A_R$ and the generators $x_i$.

$\tilde{\Omega}^i_{\Ad_K/K}$ is the separated $i$th differential module, the
plain differential module $\Omega^i_{\Ad}$ of $\Ad$ divided out by the intersection 
$\cap_n p^n\Omega^i_{\Ad}$ and tensored with $K$. There is the usual differential
complex
$$ 0 \rightarrow \Ad_K \stackrel{d}{\rightarrow} \tilde{\Omega}^1_{\Ad_K/K}
 \stackrel{d}{\rightarrow} \tilde{\Omega}^2_{\Ad_K/K} \stackrel{d}{\rightarrow}
  \ldots $$
the homology groups of which give the MW cohomology groups $H^i(A_k;K)$.

If $A_k$ is of Krull dimension 1, as in our case, then $\tilde{\Omega}^i_{\Ad_K/K} = 0$ 
for all $i \ge 2$ and so $H^1(A_k;K) = \tilde{\Omega}^1_{\Ad_K/K}/d(\Ad_K)$ and all
higher cohomology is trivial.

If $F_q$ lifts to $F$ on $\Ad$ then $F$ functorially induces a $K$-linear automorphism
$F_*$ of the $H^i$, which is independent of the choice of lift, and there is a
cohomological trace formula for $\#X(\F_{q^m})$ for all $m \ge 1$ (see next section).
Furthermore, if $F_p$ lifts to a $\sigma$-semilinear map $\Fpd : \Ad \rightarrow \Ad$,
then $\Fpd$ induces a $\sigma$-semilinear automorphism $F_{p*}$ of the $H^i$ with
$F_* = F_{p*}^n$.
\bigskip

\noindent Now let $A_k$, $A_R$ refer to the hyperelliptic algebras again. The inversion
of $y$ allows Kedlaya to define a lift of $F_p$ to $\Ad$ by
$$ x \mapsto x^p\qquad y \mapsto y^p\left(1+{{Q^\sigma(x^p)-[Q(x)]^p}\over{y^{2p}}}
 \right)^{1/2}\qquad y^{-1} \mapsto y^{-p}(1+\ldots)^{-1/2} $$
The congruence $Q^\sigma(x^p) \equiv Q(x)^p$ mod $pR[x]$ means that the standard power series
expansions of $(1+\ldots)^{1/2}$ and $(1+\ldots)^{-1/2}$ converge to elements in $\Ad_K$.

In fact, Kedlaya gives the following explicit description of $\Ad$:
$$ \Ad = \left\{ \sum^{\infty}_{-\infty} S_n(x)y^n : \mbox{deg}(S_n) \le d-1 \quad
   \liminf_{n \rightarrow \infty} {{v_p(S_n)}\over{n}} > 0 \quad 
   \liminf_{n \rightarrow \infty}{{v_p(S_{-n})}\over{n}} > 0 \right\} $$
where $v_p(f)$, $f \in R[x]$ is the smallest $m$ such that $f \in p^mR[x]$.

The hyperelliptic involution $\omega : x \mapsto x, y^{\pm 1} \mapsto -y^{\pm 1}$
extends to $\Ad$ (and $\AdK$) giving the direct sum decomposition
$$ \Ad = A^\dagger_+ \oplus A^\dagger_- \quad\mbox{with}\quad
A^\dagger_+ = \{\sum S_{2n}y^{2n}\},\mbox{ }
A^\dagger_- = \{\sum S_{2n+1}y^{2n+1}\} $$
and a corresponding decomposition of $H^1(A_k;K)$ into $+$ and $-$ components.
Kedlaya shows that the Monsky-Washnitzer trace formula leads to the result that
the numerator of the zeta-function of $C$ is just the characteristic polynomial
of $F_*$ on $H^1_-$ when $d$ is odd. We will show in Section~\ref{sec_zeta}
that the same analysis gives only a minor difference when $d$ is even.
\bigskip

\noindent\underline{Reduction steps in the computation of $F_*$}
\medskip

Kedlaya shows that a $K$-basis for the finite-dimensional $H^1(A_k;K)$ is given by the
$A_K$ differentials
$$ \{ x^idx/y : 0 \le i \le d-2 \} \cup \{ x^idx/y^2 : 0 \le i \le d-1 \} $$
the first set giving a basis for $H^1_-$ and the second for $H^1_+$. We come back
to this in the next section where we note that it also holds for $d$ even.
\medskip

The first stage of the algorithm consists of expanding the series for $F_{p*}(1/y)$
to sufficient $p$-adic precision. We will give a precise value for
the precision required at the end of Section 3. 

The second stage consists of applying two types of reduction to reexpress these images
as $K$-linear combinations of basis elements. The two basic relations are
$$ y^2 = Q(x) \qquad \mbox{and}\qquad dy = (Q'(x)/2y) dx $$
where the prime denotes the standard derivative.
\medskip

As $Q$ and $Q'$ are relatively prime in $k[x]$, there exist $U,V \in R[x]$ such that
$UQ+VQ'$ is $1$. Therefore, for
any $S \in R[x]$, there exist $A,B \in R[x]$ with $S=AQ+BQ'$. Then, for
$m \not= 2$,
$$ S{{dx}\over{y^m}} = A{{dx}\over{y^{m-2}}}+2B{{dy}\over{y^{m-1}}} =
   A{{dx}\over{y^{m-2}}}+\left({2\over{m-2}}\right)B'{{dx}\over{y^{m-2}}}-
   \left({1\over{m-2}}\right)d\left({{2B}\over{y^{m-2}}}\right) $$ 
This gives the first reduction type:
\begin{multline*}\fbox{\tt RednA}\qquad\qquad S{{dx}\over{y^m}} \equiv 
\left(A+\left({2\over{m-2}}\right)B'
  \right){{dx}\over{y^{m-2}}}\quad \mbox{if } S=AQ+BQ' \\ \end{multline*}
to reduce $m$ by 2 when $m > 2$. Note that in practice, we only apply
this for $\deg(S) < d$ because we begin by recursively dividing $S$ by
$Q$ (which is monic) to express $S$ as $S_0+S_1Q+S_2Q^2+\ldots$
with $S_i \in R[x]$, $\deg(S_i) < d$ and then replace $Q^i$ by $y^{2i}$. In fact,
we only divide out by $Q$ and replace by $y^2$ while this leads to negative powers
of $y$ in the expression. Note also that if $\deg(S) < d$ (in fact, if
$\deg(S) < 2d-1$), then $A$ and $B$ can be chosen as $SU$ mod $Q'$ and $SV$ mod $Q$,
so with $\deg(A) < d-1$ and $\deg(B) < d$.
 
In this way, {\tt RednA} applied recursively reduces $S(dx/y^m)$ to a $T(dx/y)$
or $T(dx/y^2)$ depending on the parity of $m$. Note also, that if the initial
$m$ was $\le 0$, then we could shift up instead by replacing a positive power
$y^{2i}$ by $Q^i$, but this case doesn't occur.
\medskip

\noindent The second reduction uses the differential equalities (for $r \ge 0$)
$$ d(x^r) = rx^{r-1}dx = rx^{r-1}Q(x)(dx/y^2)\quad \mbox{leading term
$rx^{r+d-1}$}$$
$$ d(2x^ry) = [2rx^{r-1}Q(x)+x^rQ'(x)](dx/y)\quad \mbox{leading term
$(2r+d)x^{r+d-1}$}$$
Subtracting multiples of the right hand sides of these from $T(dx/y^2)$ or
$T(dx/y)$, reduces the degree of $T$ until we are back to linear combinations
of basis elements. This will be referred to as {\tt RednB}.
\smallskip

Applying a number of {\tt RednA} followed by a number of {\tt RednB} steps
thus reduces any $S(dx/y^m)$ to a linear combination of basis elements. The
reductions can clearly introduce a power of $p$ into the denominator of the
final expression. This should be accurately estimated in order to gauge
{\it a priori} what the loss of $p$-adic precision may be and if there will
be non-integral terms at the end. We give the analysis in Section 3.
\bigskip

Stages 1 and 2 of the algorithm give an explicit $(d-1)$-by-$(d-1)$ matrix
$M$ over $R$ which represents the $\sigma$-linear transformation $F_{p*}$
on $H^1_-$ with respect to the chosen $x^i(dx/y)$ basis. Computationally,
the entries of $M$ will be finite approximations of the exact values which
are correct mod $p^N$ for some $N$ depending on the $p$-adic precision that
we carried out the stage~1 expansion to and on the loss of precision in stage 2.
The final stage is to compute the $n$th twisted power of $M$: $M^{\sigma^{n-1}}
M^{\sigma^{n-1}}\ldots M$. This gives the matrix of $F_*$ on $H^1_-$ and we
just need its characteristic polynomial, $P_F(t)$.

If $M$ is $p$-integral, $P_F(t)$ will be correct mod $p^N$ and the Weil bound
tells us how large $N$ should be taken for this to determine the numerator
of the zeta function of $C$. If $M$ is non-integral, it is hard to give good
{\it small} estimates of the $p$-adic precision lost in the
twisted powering. Therefore, it is highly desirable to have a $p$-integral
$M$. As we show in Section 3, for small $p$, the $x^i(dx/y)$ basis
will usually lead to $M$ with denominators whereas the $x^i(dx/y^3)$ pseudo-basis
never does.
\bigskip

\section{Adaptation of the basic algorithm}

In this section we describe in detail the adaptations to Kedlaya's algorithm outlined
in the introduction and previous section, and provide proofs of correctness.
\medskip

\subsection{Zeta function formula: even or odd case}
\label{sec_zeta}  
$\quad$
\medskip

\noindent Let $P_C(t)$ be the numerator of the zeta-function of $C/k$
(see, eg, App. C, \cite{H77}). The polynomial
$P_C(t) = t^{2g}+c_{2g-1}t^{2g-1}+\ldots+c_0$,
a monic polynomial over $\Z$. Its roots over $\C$, $\{\alpha_i\}$, all have
absolute value $q^{1/2}$ and this set is invariant under $\alpha \mapsto q/\alpha$.
Furthermore, if $S_r(\alpha) = \alpha_1^r+\ldots+\alpha_{2g}^r$
then
$$ \#C(\F_{q^r}) = q^r + 1 -S_r(\alpha) \qquad \forall r \ge 1 $$

\begin{lemma}
\label{lem_zeta}
The characteristic polynomial of $F_*$ acting $K$-linearly on $H^1(A_k;K)_-$ is
$P_C(t)$ when $d$ is odd, and is $(t-q)P_C(t)$ when $d$ is even.
\end{lemma}

\begin{proof}
The following argument is from \cite{Ked01} when $d$ is odd.
From the explicit description
of $\Ad$, it follows immediately that, if $B_k = k[x]_Q$ and $B_R = R[x]_Q$,
then $F_p$ lifts to $B^\dagger$ as a $\sigma$-linear map with $x \mapsto x^p$ and
$$ A^\dagger_+ \simeq B^\dagger \qquad\mbox{and}\qquad (\tilde{\Omega}^1_{\Ad/R})^+
\simeq \tilde{\Omega}^1_{B^\dagger/R}$$
as $F_p$-modules. Thus (abbreviating $H^i(A_k;K)$ to $H^i$ and using subscripts for the
$\pm$ eigenspaces), $H^0 = H^0_+$ and
$H^1_+$ are $F_*$-isomorphic to the cohomology groups for $Spec(B_k)$. This is
isomorphic to $\Prj^{a} := \Prj_k^1\backslash S$, where $S$ is the set of finite 
places corresponding to the irreducible factors of $Q \in k[x]$ and the point at 
infinity.

Essentially, the
contribution to cohomology resulting from the removal of closed points from $C$ to
get to $C^a$ is precisely matched by the removal of the corresponding points
from $\Prj^1$ in the odd case and gives the $H^0_+$ component. When $d$ is even,
as well as the Weierstrass points, we are removing 2 $k$-rational points from $C$ at
infinity which are swapped by the
hyperelliptic involution and lie over a single $k$-rational point of $\Prj^1$. This leads to
an extra eigenvalue $q$ in each of the $+$ and $-$ components of $H^1$. Formally, this
follows easily from the trace formula as we now show.
\medskip
 
The fixed-point theorem for Monsky-Washnitzer cohomology gives the following trace
formula for a general affine $X/k$ of dimension $n$ with (finite-dimensional) 
cohomology groups $H^i$:
$$ \#X(\F_{q^r}) = \sum_{i=0}^n (-1)^i {\rm Trace}_K((q^nF_*^{-1})^r|H^i)\qquad
 \forall r \ge 1 $$
Let $N_r =$ the number of roots of $Q(x)$ over $\F_{q^r}$ and $\delta = 0$ if
$d$ is odd and $1$ if $d$ is even. The MW trace formula for $C^{a}$ and $\Prj^{a}$
and Weil formula for $\#C(\F_{q^r})$ give
\begin{eqnarray*}
(C^{a})\qquad q^r - S_r(\alpha) -N_r -\delta &=& {\rm Tr}((qF_*^{-1})^r|H^0)- 
  {\rm Tr}((qF_*^{-1})^r|H^1_+)-{\rm Tr}((qF_*^{-1})^r|H^1_-)\\
(\Prj^{a})\qquad\qquad\qquad\quad q^r - N_r &=&
{\rm Tr}((qF_*^{-1})^r|H^0)-{\rm Tr}((qF_*^{-1})^r|H^1_+)
\end{eqnarray*}
Subtracting gives
$$ {\rm Tr}((qF_*^{-1})^r|H^1_-) = S_r(\alpha) + \delta\qquad\forall r \ge 1$$
which implies that the eigenvalues of $qF_*^{-1}$ on $H^1_-$ are 
$\{\alpha_i\} [\cup \{1\}]_{d\, even}$. Hence, the eigenvalues of $F_*$ are
 $\{\alpha_i\} [\cup \{q\}]_{d\, even}$.

Therefore the characteristic polynomial of $F_*$ on $H^1_-$ is $P_C(t)$,
if $d$ is odd, or $(t-q)P_C(t)$, if $d$ is even.
\end{proof}
\medskip

\subsection{Differential basis choices}
$\quad$
\medskip

\noindent We first note that Kedlaya's assertion that $ \{ x^idx/y : 0 \le i \le d-2 \} 
\cup \{ x^idx/y^2 : 0 \le i \le d-1 \} $ form a basis for $H^1$ remains true for
$d$ even.

By Thm. 5.6 of \cite{MW68}, the natural map $H^1_{dR}(C_K^{a}/K) \rightarrow
H^1(A_k;K)$ is an isomorphism, where $C_K$,$C_K^{a}$ are the hyperelliptic lifts of
$C$, $C^{a}$ to $K$ corresponding to the lift of $Q(x)$. The reductions
{\tt RednA} and {\tt RednB} on algebraic differentials show that the above set of
differentials generate $H^1_{dR}(C_K^{a}/K)$ and a similar argument shows that
no nontrivial $K$-linear sum of them is of the form $df$ for $f \in K[x,y,y^{-1}]
/(y^2-Q(x))$ [Note: any element of this algebra is a finite sum of the form
$\sum_{n=0}^N a_n(x)y^{-n}$]
\smallskip

\noindent{\it Remark.} That the given differentials form a basis also follows easily from
general de Rham theory for 
complete curves and their open affine subsets applied to $H^1_{dR}(C_K/K)$ and
$H^1_{dR}(C^{a}_K/K)$.
\medskip
 
\begin{definition}
We define two sets of differentials, $B_1$ and $B_2$.
  $$ B_1 = \{ dx/y, x(dx/y), \ldots, x^{d-2}(dx/y) \} $$
  $$ B_2 = \{ dx/y^3, x(dx/y^3), \ldots, x^{d-2}(dx/y^3) \} $$
The classes of the differentials in $B_1$ give a basis for $H^1_-$. $B_2$ is our
pseudo-basis whose classes only give a basis for $H^1_-$ when $d$ is odd, as we shall see.

For convenience, we also define $V_2$ as the $(d-1)$-dimensional $K$-vector subspace of
$\tilde{\Omega}^1_{\Ad_K/K}$ with basis $B_2$ and $\eta$ as the class map into $H^1_-$
  $$ \eta : V_2 \longrightarrow H^1_- \quad x^{i-1}(dx/y^3) \mapsto [x^{i-1}(dx/y^3)] $$

\end{definition}

\begin{lemma}
\label{p_div1}
\
\begin{itemize}
\item[(i)] (Kedlaya) Let $m > 2$, $S \in R[x]$ with $\mbox{deg}(S) \le d-1$.
Under {\tt RednA}, let
$$ S(dx/y^m) \equiv T(x)\{(dx/y)\mbox{ $m$ odd }, (dx/y^2)\mbox{ $m$ even}\}
\quad T(x) \in K[x], \mbox{ deg}(T) < d $$
then $p^{\left\lfloor \log_p(m-2) \right\rfloor}T \in R[x]$.
\smallskip
\item[(ii)] Let $S \in R[x]$ with $\mbox{deg}(S)=m \ge d-1$. Under {\tt RednB} let
$$ S(dx/y) \equiv T(x)(dx/y) \qquad  T(x) \in K[x], \mbox{ deg}(T) < d-1 $$
then $p^{\left\lfloor \log_p(2m-d+2) \right\rfloor}T \in R[x]$. If $d$ is even,
$p^{\left\lfloor \log_p(m-(d/2)+1) \right\rfloor}T \in R[x]$.
\end{itemize}
\smallskip
In either case, the $d(\sum_a^b S_r(x)y^r)$ differential giving the reduction
can be chosen with $p^uS_r(x)\in R[x]\;\forall r$ for the same $p^u$.
\end{lemma}

\begin{proof}
i) is just Lemma 2 of \cite{Ked01}. Note that in the statement of that Lemma,
$\log_p(2m+1)$ should be replaced by $\log_p(2m-1)$ (with $ m \ge 1$) and in the
proof, every $\pm m$ as the upper or lower limit of a sum should be replaced by
$\pm (m-1)$. The proof of the lemma works just as well for $d$ even or odd and
the final statement about $d(\sum_a^b S_r(x)y^r)$ above is what is actually
proven in Lemma 2.
\smallskip

ii) This is essentially Lemma 3 of \cite{Ked01} (or rather the corrected statement
in the errata, \cite{Ked03}). As Kedlaya notes, 
ii) and the statement about $d(\sum_a^b S_r(x)y^r)$ follow in the same way as part i)
(but more easily). We have that $ S(dx/y)-d(\sum_{r=0}^{m+d-1}2a_rx^ry) = T(dx/y)$, 
$d(2x^ry) = ((d+2r)x^{d+r-1}+\ldots)(dx/y)$  and the coefficient of $x^s$ in $T$ is
zero for $s \ge d-1$. Kedlaya's argument - considering formal expansions of the
differentials with respect to a local parameter at one of the points at infinity -
effectively shows that the largest power of $p$ that may occur in denominators
is the largest power of $p$ that can divide {\it one} of the $d+2r$
(rather than their product). When $d$ is even, it is only necessary to consider
divisibility of $(d/2)+r$ since $p$ is odd. 
\end{proof}
\bigskip

\noindent Any element of $\tilde{\Omega}^1_{\Ad_K/K}$ can be written uniquely in the
form $\sum^{+\infty}_{-\infty} S_n(x)y^n dx$ with $\mbox{deg}(S_n) < d$, which we
refer to as its {\it standard expansion}.
\bigskip

\begin{lemma}
\label{p_div2}
\
\begin{itemize}
\item[(i)] For all $\omega \in B_2$, the standard expansion of $F_{p*}\omega$ is of
the form $\sum_{n \geq 3} B_n(x) (dx/y^n)$.
\smallskip
\item[(ii)] {\tt RednA} on the $\sum_{n \geq 1} S_n(x) (dx/y^n)$ 
      part of the standard expansion of $F_{p*}(x^i(dx/y))$ (resp. $F_{p*}(x^i(dx/y^3))$)
       gives a form which is a linear combination of elements
      of $B_1$ (resp. $B_2$) with $p$-integral coefficients.
\smallskip
\item[(iii)] Consider the coefficients of the $B_1$ expansion resulting
      from {\tt RednB} on the $\sum_{n \geq 1} S_n(x)y^n dx$ 
      part of the standard expansion of $F_{p*}(x^{i-1}(dx/y))$.
  \begin{itemize}
  \item[(a)] If $d = 2g+1$, then these coefficients are $p$-integral
   for $i \le g$ and for $i = g+r$ have denominator bounded by
   $p^{-\left \lfloor \log_p(2r-1) \right \rfloor}$.
  \item[(b)] If $d = 2g+2$, then these coefficients are $p$-integral
   for $i \le g+1$ and for $i = g+r+1$ have denominator bounded by
   $p^{-\left \lfloor \log_p(r) \right \rfloor}$.
   \end{itemize}
\end{itemize}
\smallskip

\noindent By part (i), we can use {\tt RednA} to reduce
    $F_{p*}\omega$ back to linear combinations of elements in $B_2$
    rather than descending to $B_1$. This is what is meant in
    part (ii). In this way, we get a
    $\sigma_p$-linear map (also denoted $F_{p*}$) $V_2 \rightarrow V_2$.

\end{lemma}

\begin{proof}
We have, for $1 \le i \le d-1$, $k = 0$ or $1$,
\begin{eqnarray}
F_{p*}(x^{i-1}(dx/y^{2k+1}))
    & = & {\textstyle
      x^{p(i-1)}  y^{-(2k+1)p}
    \left ( 1 + p
    \left ({\frac{Q^{\sigma_p}(x^p)-(Q(x)^p)}{p}} \right )y^{-2p} 
           \right )^{-(2k+1)/2} d(x^p) } \nonumber\\
    & = & {\textstyle
        p x^{pi-1} y^{-(2k+1)p}
         \left ( 1 + p\{a_1(x)y^{-2} + \ldots + a_{p}(x)y^{-2p}\}
	 \right )^{-(2k+1)/2} dx }\nonumber \\
    & = & {\textstyle
        p x^{pi-1} y^{-(2k+1)p}
         \left ( 1 + \sum_{n=1}^\infty {-(2k+1)/2 \choose n}p^n
	 \{ \ldots \}^n
	 \right ) dx }\nonumber \\
    & = & {\textstyle
        p x^{pi-1}
         \left ( \sum_{m \, odd \, \ge (2k+1)p}
           p^{\left \lceil {\frac{m-p}{2p}} \right \rceil - k }
	 \ b_m(x) y^{-m}
	 \right ) dx }\nonumber
\end{eqnarray}
with $a_i(x),b_i(x) \in R[x]$ of degree less than $d$.
Note that $b_{(2k+1)p}(x) = 1$ and that  $\{\ldots\}^n$ when expanded is then reduced
to the form $A_1(x)y^{-2}+\ldots+A_{pn}(x)y^{-2pn}$ with $A_i(x) \in R[x]$ of degree
less than $d$.

When we multiply each term in the final sum by $x^{pi-1}$ and reduce
using the relation $y^2 = Q(x)$, we see that the
result is
$$ F_{p*}(x^{i-1}(dx/y^{2k+1})) = {\textstyle
     \sum_{m \, odd \, \ge m_0} c_m(x) y^{-m} dx} $$
where
\begin{equation}\label{eqn:m_size}
   m_0 \ge (2k+1)p - 2\lfloor (pi-1)/d \rfloor
\end{equation}
and each $c_m(x) \in pR[x]$. Here we have used $b_{(2k+1)p}(x) = 1$ to get $pi-1$ rather
than $pi+d-2$. Furthermore,
\begin{equation}\label{eqn:c_val} c_m(x) \in 
   p^{\left \lceil {\frac{m-p}{2p}} \right \rceil + 1 - k } R[x]\qquad
   \forall m \ge (2k+1)p
\end{equation}
\smallskip

(i) When $k=1$, by (\ref{eqn:m_size}) with $i = d-1$, $m_0 \ge p+2 > 3$.
\medskip

(ii) First note that for $m < (2k+1)p \le p^2$, $\log_p(m-2) < 2$.
From Lemma \ref{p_div1} and (\ref{eqn:c_val}), we see
that it suffices to prove that
$$ \left \lceil {\frac{m-p}{2p}} \right \rceil + 1 - k -
      \lfloor \log_p(m-2) \rfloor \ge 0 \quad \forall m \,
      odd \, \ge (2k+1)p $$
      
For $k=0$, the inequality with the floor and ceiling brackets removed
holds for $m > 2p+1$ by elementary calculus. For $ p \le m \le 2p+1$,
it is clear.

For $k=1$ and $p \ge 5$, the inequality again holds for $ m \ge 5p$ by
calculus and for $3p \le m < 5p$ it is clear.

For $k=1$ and $p=3$, the inequality holds for $ m \ge 3p^2+1$ by
calculus and for $3p \le m < 3p^2+1$ it is again easy to
check directly.
\medskip

(iii) Consider the $px^{pi-1}p^\alpha b_m(x)y^{-m}$ terms that give contributions
to the $\sum_{n \ge 1}$ sum. Expressing $x^{pi-1}b_m(x)$ as $u_r(x)y^{2r}+\ldots
u_0(x)$ with $\deg(u_i) < d$, we must have $r \ge (m-1)/2$ and the contribution
will be expressible in the form $S(dx/y)$ with $\deg(S) = pi-1+\deg(b_m)-d(m-1)/2$.
This last expression must be greater than or equal to $d-1$ for non-trivial reduction 
under {\tt RednB}.
For such $m$, writing $d_m$ for $\deg(b_m)$,
the above and Lemma \ref{p_div1} (ii)
show that the power of $p$ in the denominator of the {\tt RednB} reduction of the
contribution from the index $m$ term is bounded above by
$$\lfloor \log_p(2pi-md+2d_m)\rfloor - 1 - \lceil (m-p)/2p\rceil\qquad
 \mbox{ if } d=2g+1 $$
$$\lfloor \log_p(pi-m(d/2)+d_m+1)\rfloor - 1  - \lceil (m-p)/2p\rceil\qquad
 \mbox{ if } d=2g+2 $$

We have that $d_p = 0$ ($b_p(x)=1$) and $d_m \le d-1$ for $m \ge p+2$. Since $m \ge p$
is odd, the above expressions are maximal when $m=p$.
(a) and (b) follow easily from this.
\end{proof}
\bigskip

The bounds in Lemma \ref{p_div2} (iii) for denominators in the reduction of 
$F_{p*}(x^{i-1}(dx/y))$ are sharp. The proof shows that the first term in the
power series expansion $px^{pi-1}(dx/y^p)$ is the only one that can contribute
to the given maximal power of $p$ and
for a general $Q$ it does indeed lead to denominators equal to the bounds.

Thus, as is readily confirmed in practice by computer computations, we reach the 
following
\bigskip

\noindent\underline{Conclusion:} When $d=2g+1$ and $p > 2g-1$ or $d=2g+2$ and $p > g$,
the transformation matrix $M$ for $F_{p*}$ w.r.t. basis $B_1$ for $H^1_-$ is
$p$-integral. When these equalities for $p$ do not hold however, for a general $Q$,
entries in the lower rows of $M$ have powers of $p$ in the denominator
given by the bounds in the last part of Lemma \ref{p_div2}. 
\bigskip

On the other hand, Lemma \ref{p_div2} shows that {\tt RednA} applied to
$F_{p*}(\omega)$ for $\omega \in B_2$ reduces back to an expression that is always
an $R$-linear combination of the elements of $B_2$, so formally leads to a
$p$-integral transformation matrix $M$.

If $B_2$ gives a basis for $H^1_-$, then this $M$ genuinely represents $F_{p*}$
on that space and $B_2$ can replace $B_1$ as the chosen basis for computations.

Even when $B_2$ doesn't give a basis, this $M$ can still be used. The
above shows that the kernel of $\eta$ and its image in $H^1_-$ are $F_{p*}$- and
hence also $F_*$-stable. 

The following result will be demonstrated in the next section.
 
\begin{proposition}
\label{prop_eta}
\
\begin{itemize}
\item[(i)] $\eta$ is an isomorphism when $d=2g+1$ but has a 1-dimensional kernel and
cokernel when $d=2g+2$.
\smallskip
\item[(ii)] In the latter case, $F_*$ is the identity on $\ker(\eta)$
and acts as multiplication by $q$ on $H^1_-/\myim(\eta)$.
\end{itemize}

\end{proposition}
\medskip

This justifies the adaptation of Kedlaya's algorithm given in Section 2, which always
works with a $p$-integral $M$. In summary:
\bigskip

\noindent\underline{New Algorithm}
\begin{itemize}
\item $d=2g+1$. If $p \ge 2g$ then the algorithm is unchanged. If $p < 2g$ then the
algorithm is as before, but use differential basis $B_2$ instead of $B_1$.
\item $d=2g+2$, $p > g$. Apply the algorithm as for odd $d$ with differential
basis $B_1$. At the end, remove a factor $t-q$ from the characteristic polynomial of
$F_*$.
\item $d=2g+2$, $p \le g$. Formally apply the algorithm as for odd $d$ with 
pseudo-basis $B_2$. At the end, remove a factor $t-1$ from the characteristic 
polynomial of $F_*$.
\end{itemize}
\bigskip

\noindent\underline{Efficiency} If $N_1 = \lceil (ng/2)+\log_p(2{2g\choose g}) \rceil$
($q=p^n$) and $N = N_1+\lfloor\log_p(2N_1)\rfloor + 1$, then estimates using Lemma
\ref{p_div2} and the Weil bound show that it suffices to compute $(1+(Q^\sigma(x^p)-
Q(x)^p)y^{-2p})^{-(2k+1)/2}$ to accuracy $p^N$ in order that $M$ is of sufficient
$p$-adic accuracy to determine $P_C(t)$. Here, $k=0$ if we use $B_1$ and $k=1$ for
$B_2$. Using $k=1$ rather than $k=0$ makes virtually no difference in computational
efficiency here, and the reduction of $F_{p*}(\omega)$ back to a linear combination
of basis elements is in fact slightly better when using $B_2$.

However, $d=2g+2$ rather than $2g+1$ does increase the size of the bases by 1 element,
meaning that one extra reduction of a $F_{p*}(\omega)$ has to be performed. Also the
$(d-1) \times (d-1)$ matrix $M$, which has to be $\sigma$-powered to the $n$th power,
has an extra row and column. This does make a small difference (more so for smaller
$g$), which makes it worth looking for a $k$-rational root of $Q(x)$ and moving that to
$\infty$ to transform to $d=2g+1$. In general, though, no such transformation is
possible.
\bigskip

\section{Proof of Proposition \ref{prop_eta}}
\medskip

\noindent Proposition~\ref{prop_eta} of the last section on the $\eta$ map is proven in the 
following three lemmas. 
\medskip

\begin{lemma}
\label{lb2_1}
\
If $d = 2g+1$ then $\eta$ is an isomorphism onto $H^1_-$.

If $d=2g+2$ then $\eta$ has a one dimensional kernel generated by
$V(dx/y^3)=d(-2S/y)$ where $V = SQ'-2S'Q$ and $S = x^{g+1}+\ldots \in K[x]$
is the unique monic degree $g+1$ polynomial such that $V$ is of degree $\le 2g$.
\end{lemma}

\begin{proof}
Using the fact that $B_1$ is a basis for $H^1_-$ and the {\tt RednA} formula,
we see that an element of the kernel of $\eta$ corresponds to a differential
of the form $V(dx/y^3)$ with $\deg(V) \le d-2$ and $V=SQ'-2S'Q$.

If $V = a_rx^r+\ldots$ with $r \ge 0$, $a_r \neq 0$, then the leading term of
$SQ'-2S'Q$ is $(d-2r)a_rx^{d+r-1}$, so $d$ must equal $2r$. So, $d=2g+2$
and $r = g+1$. Normalising $S$ so that its leading coefficient is $1$, it follows
easily that its lower coefficients are completely
determined by the condition on $\deg(V)$. Explicitly, if $b_i$ is the coefficient of
$x^i$ in $S$, then the condition that the coefficient of $x^{d+i-1}$ in  $SQ'-2S'Q$
is zero, $0 \le i \le g$, translates into
$$ (2g+2-2i)b_i = \mbox{ some linear combination of $b_j, j \ge i+1$}$$
This determines the $b_i$ inductively and gives a unique $S$ and  $V$ up to $K$-scaling.
\end{proof}
\medskip

\begin{lemma}
\label{lb2_2}
\
When $d=2g+2$, $F_*$ acts trivially on $\ker(\eta)$.
\end{lemma}

\begin{proof}
From the last lemma, $\ker(\eta)$ is 1-dimensional and generated by $\omega =
V(dx/y^3)=d(-2S/y)$ with $S = x^{g+1}+\dots$. As $\ker(\eta)$ is $F_*$-stable,
$\omega$ is an eigenvector for $F_*$ with eigenvalue $\lambda$, say. We must
show that $\lambda=1$.
\medskip

Considering the images in $H^1_-$ and using Lemma \ref{p_div2} (ii), we get
$$ F_*(\omega) = \lambda\omega - 2d(f) \quad f = \sum_{r=1}^\infty {B_r(x) \over
y^{2r+1}} \in (A^\dagger_K)^-
 \Rightarrow\quad d(F\left({S\over y}\right)) = \lambda d\left({S\over y}\right)+
d\left({B_1 \over y^3}+{B_3 \over y^5}\ldots\right)$$
So
\begin{equation}\label{eqn:Feqn}
   F\left({S\over y}\right) = \lambda\left({S\over y}\right) +
   \left({B_1 \over y^3}+{B_3 \over y^5}\ldots\right) \in (A^\dagger_K)^-
\end{equation}
In fact, this equality is true up to addition of a constant in $K$, but as both
sides are in the $-$ eigenspace, the constant must be zero. The $B_i$ here have
degree $< d$.
\medskip

Now, as in the proof of Lemma \ref{p_div2}, we see that if the standard expansion
of $f \in A^\dagger_K$ is of the form $\sum_{n \ge 3} a_n(x)/y^n$, then $F_p(f)$
has the same property.

Also, expanding $S(x^p)$ as $u_r(x)Q(x)^r+\ldots +u_0(x) = u_r(x)y^{2r}+\ldots +u_0(x)$
with $\deg(u_i) < d$, we easily get that $r=(p-1)/2$ and $u_r(x) = x^{g+1}+\dots$.

Then, using $F_p(1/y) = y^{-p}(1+a_2(x)/y^2+a_4(x)/y^4+\ldots)$, we find that
$F_p(S/y) = S_1(x)/y + b_3(x)/y^3 + \ldots$ with $S_1(x) = x^{g+1}+\ldots$.
Iterating, we see that the same holds for $F(S/y)$. Then, (\ref{eqn:Feqn}) implies
that $\lambda = 1$.
\end{proof}
\medskip

\begin{lemma}
\label{lb2_3}
\
When $d=2g+2$, $F_*$ acts on $H^1_-/\myim(\eta)$ as multiplication by $q$.
\end{lemma}

\begin{proof}
We already know that $\myim(\eta)$ is an $F_*$-stable codimension 1 subspace of
$H^1_-$ and that the eigenvalues of $F_*$ on $H^1_-$ are $q$ and the roots of
$P_C(t)$, the numerator of the zeta-function of $C$. We need to show that
the eigenvalues of $F_*$ on $\myim(\eta)$ are precisely these latter roots.

We will prove the lemma by using an isomorphism to an odd degree model over
an extension $\F_{q^r}$ of $k$ where $Q \in k[x]$ has a root. In fact, replacing
$F$ by $F^r$ corresponds to replacing the basefield $k=\F^q$ by $k_1=\F^{q^r}$
and the roots of $P_{C/k_1}(t)$ are the $r$th powers of the roots of $P_{C/k}(t)$.
These latter roots have absolute value $q^{r/2}$ in every complex embedding whereas
$q^r$ obviously has absolute value $q^r$. So we can assume that $Q$ has a root in $k$.
\medskip

First note that
$$ \myim(\eta) = \{ \omega \in H^1_-\quad |\quad \mbox{Residue}_{\infty_1}(\omega) =
 \mbox{Residue}_{\infty_2}(\omega) = 0 \} $$
as both sides of the equality have codimension $1$ in $H^1_-$ and the LHS lies
in the RHS (in fact, all differentials of the form $x^i(dx/y^3)$, $i \le d-2$
are holomorphic at both points at infinity).
\medskip

We can translate a root of $Q(x)$ to zero by a $x \mapsto x-\alpha$ translation
(this changes the lift of $F$ but not $\myim(\eta)$), so assume that $Q(x) = x^{2g+2}
+a_{2g+1}x^{2g+1}+\ldots +a_1x \in k[X]$, $a_1 \neq 0$.

Let $\tilde{Q}(X) = X^{2g+1}+(a_2/a_1^2)X^{2g}+\ldots +(1/a_1^{2g+2})$.

The equation $Y^2 = \tilde{Q}(X)$ defines a new smooth, odd-degree affine model
for $C$ and we have
$$ B_k := {k[X,Y,Y^{-1}]\over(Y^2-\tilde{Q}(X))} \hookrightarrow A_k = 
 {k[x,y,y^{-1}]\over(y^2-Q(X))} \quad X \mapsto 1/(a_1x),\; Y \mapsto 
y/(a_1x)^{g+1}$$
[note: $1/(a_1x) = (1/(a_1y^2))(a_1+a_2x+\ldots) \in A_k$].
Letting $B^\dagger$ be the smooth lift of $B_k$ corresponding to the lift to
$R[X]$ of $\tilde{Q}$ with the coefficient lift compatible with that of $Q$,
we get the corresponding commutative diagram
$$ \begin{CD}
  B^\dagger @>>> A^\dagger\\
  @V{F^{(1)}}VV   @VV{F^{(2)}}V\\
  B^\dagger @>>> A^\dagger
\end{CD} $$
for some choice of $q$-Frobenius lifts $F^{(1)}$ and $F^{(2)}$. All maps commute
with the automorphisms induced by the hyperelliptic involution.
\medskip

One easily sees that $A_k = B_k[1/X]$. The Main Theorem of \cite{Mon68} implies that
$$ H^1(B_k;K) \hookrightarrow H^1(A_k;K) = H^1$$
with image the $K$-subspace of elements with residues $0$ at $\infty_1$ and
$\infty_2$, the images of points with $X = 0$ under the automorphism of $C$
induced from $B_k \hookrightarrow A_k$. [In fact, a bit of computation verifies the
residue condition directly from the explicit maps].
\medskip

Thus $\myim(H^1(B_k;K)^-) = \myim(\eta)$ and as we know that the eigenvalues of $F_*$
on $H^1(B_k;K)^-$ are the roots of $P_C(t)$ (the odd degree case), the result follows.
\end{proof}

\bibliographystyle{amsalpha}
\bibliography{mike.bib}
\end{document}